\documentclass[12pt, leqno]{article}
\usepackage{amsmath, amsthm, enumitem, thmtools}
\usepackage{amsfonts}
\usepackage{amssymb}
\usepackage[usenames,dvipsnames]{color}
\usepackage{graphicx}

\usepackage{tcolorbox}
\usepackage{wrapfig}
\definecolor{fig}{RGB}{76,78,92}
\usepackage[labelfont=bf,font={color=fig,small}]{caption}
\usepackage[colorlinks,citecolor=cyan]{hyperref}
\usepackage{tikz}
\usetikzlibrary{intersections}
\usetikzlibrary{arrows,snakes,shapes}
\usepackage{pgfplots}
\usepgfplotslibrary{fillbetween}
\usepackage{enumitem}
\usepackage{tabularx}
\usepackage{array}
\usepackage{colortbl}
\tcbuselibrary{skins}
\usepackage{setspace}
\usepackage{smartdiagram}
\usepackage{enumitem}
\usepackage{multicol}
\usepackage{wrapfig}
 \usepackage{textcomp}
 \usepackage{etoolbox}
\usepackage{microtype}
\usepackage{bbm}
\usepackage[toc]{appendix}
\usepackage{bm}
\usepackage{listings}
\usepackage{nicefrac} 
\usepackage{xfrac}    
\usepackage[utf8]{inputenc}
\usepackage[T1]{fontenc}
\usepackage{hyperref}

\usepackage{color,calc}

\makeatletter
\definecolor{shade}{gray}{0.8}
        {
          \raggedright
        \setlength{\rightmargin}{\leftmargin}
        \setlength{\itemsep}{-12pt}
        \setlength{\parsep}{20pt}
        \begin{lrbox}{\@tempboxa}%
        \begin{minipage}{\linewidth-2\fboxsep}
        }%
        {
        \end{minipage}%
        \end{lrbox}%
        \fcolorbox{black}{shade}{\usebox{\@tempboxa}}\newline
        }%
\makeatother

\newtheorem{theorem}{Theorem}
\newtheorem{lemma}{Lemma}
\newtheorem{cor}{Corollary}
\newtheorem{conjecture}{Conjecture}

\renewcommand{\eqref}[1]{\hyperref[#1]{(\ref*{#1})}}

\newcommand*{\norm}[1]{\lVert #1 \rVert}

\newtheorem{definition}{Definition}

\newcommand*{\pref}[1]{\hyperref[#1]{(\ref*{#1})}}
\newcommand*{\refpref}[2]{\hyperref[#2]{\ref*{#1}(\ref*{#2})}}

\usepackage{xspace}

\newcommand{\dd}{{\mathrm d}}

\title{The Replicator Coalescent}

\author{A. E. Kyprianou\thanks{
Department of Statistics
University of Warwick
Coventry
CV4 7AL, UK. E-mail: \texttt{andreas.kyprianou@warwick.ac.uk}}
, \ L. Pe\~{n}aloza\thanks{Instituto de Investigación de Matem\'aticas y Actuaría, 
Universidad del Mar, campus Huatulco. Mexico. Email: \texttt{lizbeth@huatulco.umar.mx}},    
 and 
 T. Rogers\thanks{Department of Mathematical Sciences, University of Bath, Claverton Down, Bath, BA2 7AY. UK. Email: \texttt{ma3tcr@bath.ac.uk}}
 }

\begin{document}

\maketitle
\vspace{-0.5cm}
\begin{abstract}
 We consider a stochastic model, called the {\it replicator coalescent}, describing a system of blocks of $k$ different types which undergo pairwise mergers at rates depending on the block types: with rate $C_{ij}\geq0$ blocks of type $i$ and $j$ merge, resulting in a single block of type $i$. The replicator coalescent can be seen as generalisation of Kingman's coalescent death chain in a multi-type setting, although without an underpinning exchangeable partition structure. The name is derived from a remarkable connection  between the instantaneous dynamics of this multi-type coalescent when issued from an arbitrarily large number of blocks, and the so-called {\it replicator equations} from evolutionary game theory. By dilating time arbitrarily close to zero, we see that initially, on coming down from infinity, the replicator coalescent behaves like the solution to a certain replicator equation. Thereafter, stochastic effects are felt and the process evolves more in the spirit of a multi-type death chain.
\medskip

 \noindent{\bf Key words:} Markov chain, coalescent, coming down from infinity.

\medskip

 \noindent{\bf MSC 2020:}  60J80, 60E10.
\end{abstract}

\section{Introduction} 
In this article, we are interested in developing a multi-type analogue of Kingman's coalescent as a death chain, called a {\it replicator coalescent}, with the following interpretation. Blocks take one of $k$ different types. Mergers within blocks may take place as well as mergers of blocks with two different types. In the latter case, we will need to specify what type the two merging blocks of different type will take.  To this end, let us introduce the $k\times k$ rate matrix $\bm C = (C_{i,j})$ the merger rate matrix with $C_{i,j}\geq 0$ for all $i,j\in\{1,\cdots, k\}$. This matrix (which is not an intensity matrix) encodes the evolution of a continuous-time Markov chain, say $({\bm n}(t), t\geq0) $, on the state space
\[
\mathbb{N}^k_* = \left\{\bm \eta \in \mathbb{N}^k_0: \sum_{i=1}^k\eta_i\geq 1\right\},
\]
 where $\mathbb{N}_0 = \{0,1,2, \cdots\}$, 
 which is defined in the following way. 
 Given that ${\bm n}(t) = (n_1  , \cdots,  n_k )\in \mathbb{N}^k_*$ such that $\sum_{i =1}^k n_i > 1$:
\begin{itemize}
\item For $i\in\{1,\cdots,k\}$ any two specific blocks of type $i$ will merge at rate $C_{i,i}$, and hence a total merger rate of type $i$ blocks equal to 
$
C_{i,i}{n_i  \choose 2}.
$
\medskip

\item For $i\neq j$, both selected from  $\{1,\cdots, k\}$, any block of type $i$ will merge with any block of type $j$, producing a single block of type $i$, at rate $C_{i, j}$. The total rate of events of this type is thus $C_{i,j} n_in_j$.
\end{itemize}

\medskip

We can interpret $({\bm n}(t), t\geq 0) $  as the  evolution of the population of $k$ types that exhibit both inter-type and intra-type competition. The rate $C_{i,i}$ is the rate at which two individuals of type $i$ compete for resource resulting in one of them not surviving. Moreover, at rate $C_{i,j}$  individuals of type $i$ and $j$ encounter one another in a competition for resource, resulting in $j$ not surviving. In this respect, our replicator coalescent echos features of the so-called O.K. Corral model describing a famous 19th Century Arizona shoot-out between lawmen and outlaws in \cite{King, Vking, Vking2} as well as the (birth-)death process in \cite{BEG}. An example of the sample path of the process $\bm n$ is given in Figure \ref{fig1} in the setting $k = 3$.
\begin{figure}[h!]
  \begin{center}
\includegraphics[height=10cm]{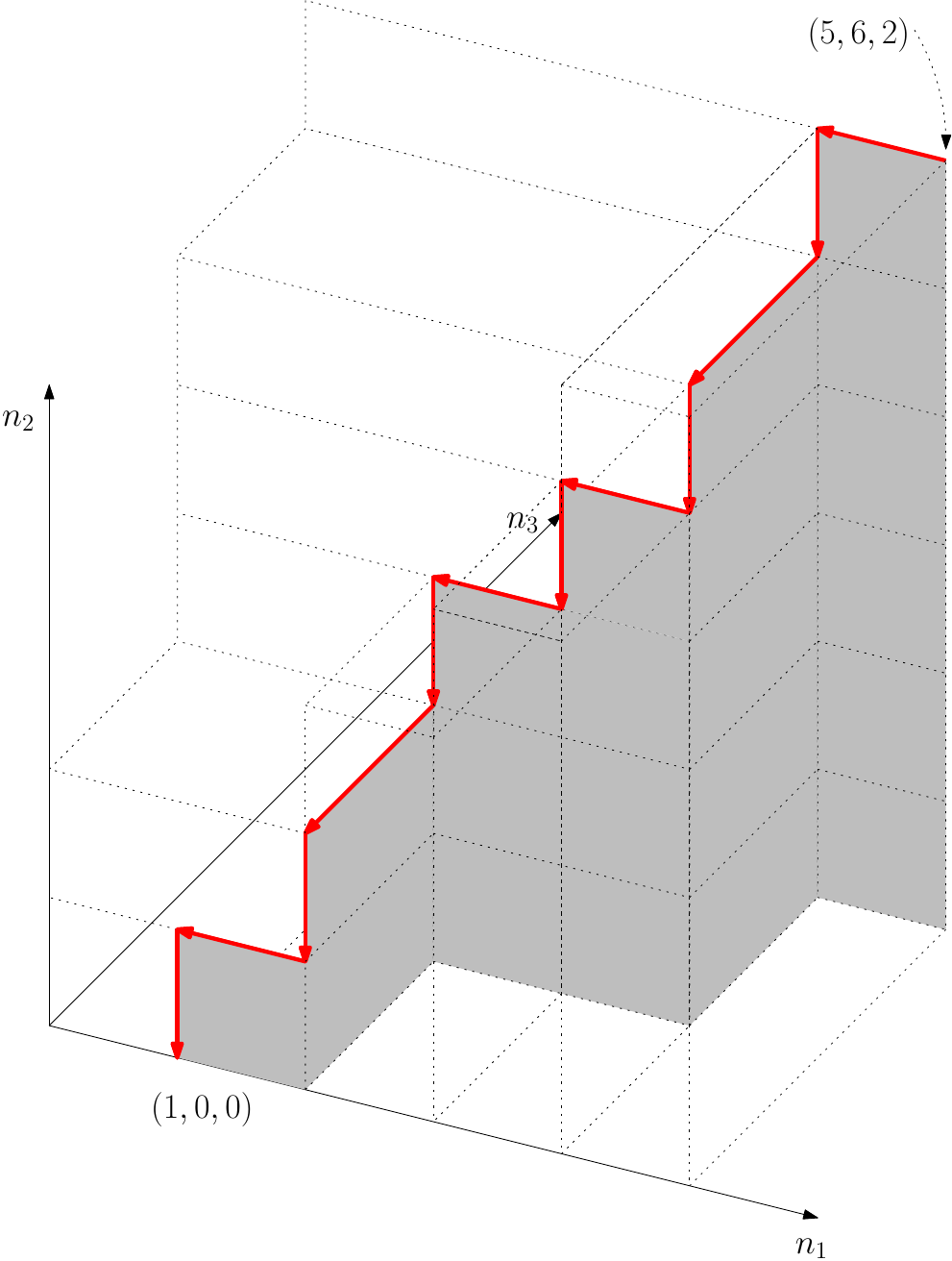}
  \end{center}
\caption{\it A path of a replicator coalescent block numbers with $k = 3$, initiated from $\bm n(0) = (5,6,2)$ and reducing to a population of one with $\bm n(\gamma_1) = (1,0,0)$. The diagram represents the range of the process and there is no time axis.}
  \label{fig1}
  \end{figure}

The reader will note that the rate at which inter-type mergers occur is  that of Kingman's coalescent.  When there is only one type, and hence only inter-type coalescence occurs, 
the replicator coalescent is therefore nothing more than the death chain of a Kingman coalescent. 
In this sense, the process $(\bm n, \mathbb{P})$ may be thought of as a multi-type variant of the Kingman death chain. 

We also note that the requirement $C_{i,j}>0$, for all $i,j$, is a sufficient  condition to ensure that 
populations of different types are able to interact with one another, which, in turn, will allow for the population to collapse to a single surviving individual. 
If, for example we were to take $C_{i,j}=0$ for all $i\neq j$, then we have $k$ independent Kingman coalescence processes that do not interact, in which case, the absorbing state of $(\bm{n}(t),t\geq 0)$ is the vector $(1,\cdots, 1)$. 
That said, even sticking to the requirement that the model ensures that the population reduces to a single surviving individual, the condition $C_{i,j}>0$, for all $i,j$,  may well be replaced by  a weaker `irreducibility'-type condition. Nonetheless, we  refrain from exploring this further at this point as, later on in this article, we will require more conditions on the matrix $C_{i,j}$, for different reasons, which will supersede the current discussion.

\medskip
The specific structure of the replicator coalescent does not permit an interpretation in terms of exchangeable partition structures as is the case for Kingman's coalescent, neither when considering the total population, nor when considered as a vector valued process. In the former case, this is obviously because blocks are subject to different rates according to their type and therefore cannot be exchangeably labelled. In the latter case, a notion of multi-type exchangeability is possible and was discussed in the context of coalescence in \cite{JKR}. Unfortunately, the way in which mergers  occur across different types of blocks is just outside the definition given in \cite{JKR}, which insists on a random selection of multiple mergers, which cannot be arranged to be a single merger via parameter choices.

\medskip

Although the replicator coalescent lives in the space $\mathbb{N}^k_*$, we prefer to describe it via a so-called
{\it $L^1$-polar decomposition} in the spirit of e.g. \cite{BEG}. To this end,  define 
\[
\sigma(t) = ||\bm n(t)||_1 = n_1(t)+\cdots+n_k(t)\in \mathbb{N}
\]
 and let 
 \[
 \bm  r(t) = \arg(\bm n(t)): = \sigma(t)^{-1}\bm{n}(t)\in\mathcal{S}^k,
 \]
  where 
  \[
  \mathcal{S}^{k}:=\left\{ (x_1,x_2,...,x_{k})\in\mathbb{R}^k:\sum_{i=1}^{k}x_i=1 \mbox{, } x_i\geq0 \mbox{ }\forall i\right\}
  \]
   is the $(k-1)$-dimensional simplex, with vector entries $\bm{r}_i(t) = \sigma(t)^{-1}n_i(t)$, $i = 1,\cdots, k$. We will additionally write and occasionally use $ \mathcal{S}^{k}_+$ to have the same definition as $ \mathcal{S}^{k}$, albeit each of the $x_i>0$.

\medskip

We often refer to the process $\bm n$ as $( \bm r, \sigma^{-1})$. In particular, if $\bm\eta =(\eta_1,\cdots,\eta_k)\in \mathbb{N}^k_*$,  then, we will use 
$\mathbb{P}_{\bm\eta}$
 for the law of the replicator coalescent issued from state $\bm n(0) = \bm \eta$. The usual convention would be to think of  the family of probabilities  $\mathbb{P}= (\mathbb{P}_{\bm\eta}, \bm\eta\in \mathbb{N}^k_*)$, however we interchangeably also think of $\mathbb{P}= (\mathbb{P}_{\bm\eta}, \bm\eta\in \mathcal{S}^k\times \mathbb{N}^{-1})$, where $\mathcal{S}^k\times \mathbb{N}^{-1}: = \{(\bm x, 1/n): \bm x\in \mathcal{S}^k\text{ and }n\in \mathbb{N}\}$.
 

\medskip

In the setting of the block-counting process for Kingman's coalescent, there are three fundamental facts which are now taken for granted in mainstream literature. Firstly Kingman's coalescent block-counting process  comes down from infinity almost surely. Secondly, it comes down from infinity in such a way that the number of blocks divided by $1/t$ converges to a constant as $t\to0$. Thirdly, and somewhat trivially, the block counting process is a death chain with  an absorbing state which is a single block.
This inspires us to address  the following three main questions of our replicator coalescent:

\begin{itemize}
\item[(1)] Does it `come down from infinity' in an appropriately prescribed sense?
\item[(2)] What is the distribution on $\{1,\cdots, k\}$ of the type of the terminal block?

\end{itemize}

We are interested in characterising the behaviour of the replicator coalescent as we start it from an initial population that `tends to infinity' in a prescribed way, and as such we will give a response to (1).
In doing so, we will unravel a remarkable connection with the theory of evolutionary dynamical systems, described by so called {\it replicator equations}, hence our choice of the name {\it replicator coalescent}. Our response to (1) is by no means a complete story. For example we  don't show the existence of entrance laws on Skorokhod space, but rather we focus on the behaviour of the process as we limit its initial state to a boundary state `at infinity', which means an initial condition for $( \bm r, \sigma^{-1})$ in $\mathcal{S}^k_+\times \{0\}$.

\begin{definition}
Henceforth we will say that `$(\bm \eta^N, N\geq 1)$  tends to $( \bm r_0, 0)\in \mathcal{S}^k_+\times \{0\}$', if $\bm \eta^N\in  \mathbb{N}^k_*$ such that $\norm{\bm\eta^N} = N$ and $\arg(\bm\eta^N)\to \bm r_0$ as $N\to\infty$.
\end{definition}

We are unable to provide any results for (2) and believe this to be an extremely difficult problem; even in light of related results e.g. on the aforementioned OK Coral model in \cite{King, Vking, Vking2}. This short article is but an introduction to replicator coalescence, offering the opportunity for further analysis to take place. Indeed, in future work, we aim to give a more precise statement on the convergence on the Skorokhod space of the process to a unique entrance law which exhibits continuity at time zero. We comment further on this in the final section of this paper.
\section{Main results}

For our first result, we show that the replicator coalescent comes down from infinity in a relatively specific sense. We study the time $\gamma_m = \inf\{ t> 0 : \sigma(t) \leq m\}$, $m\in\mathbb{N}$, that the coalescent first reaches a state with $m$ blocks in total, which can be bounded in probability for large $N$ according to the following result.

\begin{lemma}\label{comingdownfrominfinity}
Without any further requirement on $C_{i,j}$, there exists a Kingman coalescent death chain $\nu^-$ such that $\sigma\geq \nu^-$ under $\mathbb{P}_{\bm\eta}$, for each ${\bm \eta} = (\arg({\bm n}(0)), \sigma(0)^{-1})\in \mathcal{S}^k\times \mathbb{N}^{-1}$. 
Conversely, if $C_{i,j}>0$ for all $i,j$, there exists a Kingman coalescent death chain $\nu^+$ such that $\sigma\leq \nu^+$ under $\mathbb{P}_{\bm\eta}$, for each ${\bm \eta} = (\arg({\bm n}(0)), \sigma(0)^{-1})\in \mathcal{S}^k\times \mathbb{N}^{-1}$. 

In particular, under the  assumption that $C_{i,j}>0$ for all $i,j$,
the replicator coalescent comes down from infinity  in the sense that, for all $\varepsilon>0$, 
\[
\lim_{m\to\infty}\lim_{N\to\infty}\mathbb{P}_{\bm\eta^N}(\gamma_m<\varepsilon) = 1.
\]
\end{lemma}

As it comes down from infinity, the standard Kingman coalescent with merger rate $c>0$ has block count $(\nu(t), t\geq0)$ that is approximately described by the ordinary differential equation
\begin{equation}
\dot \nu(t)=-c\nu(t)^2/2.
\label{Kode}
\end{equation}
  It turns out that the corresponding ODE for our coalescent is already known in the evolutionary game theory literature as the {\it replicator equation}.
\medskip

Replicator equations are used to describe a population of $k$ types, with the proportion of the total population of type $i\in\{1,\cdots, k\}$ at time $t\geq0$ denoted $x_i(t)\in[0,1]$. These values sum to one, so $\bm{x}(t): = (x_1(t),\cdots, x_k(t))$ lives in $\mathcal{S}^k$.  The replicator equations are then written as
  \[ 
  \dot x_i(t)=x_i(t)(f_i(\bm{x}(t))-\overline{f}(\bm{x}(t))), \qquad i = 1,\cdots, k,\,  t\geq0,
  \]
where $f_i:\mathcal{S}^k\mapsto \mathbb{R}$ describes the ``fitness'' of type $i$ as a function of the current population density and  $\overline{f}(\bm{x}(t))=\sum_{i=1}^nx_if_i(\bm{x}(t))$ is the average population fitness. 

\medskip

Fitness is often assumed to depend linearly upon the population distribution, with coefficients organised in the ``payoff matrix'' $A$. Specifically, let $A_{i,j}$ denote the payoff for a player of type $i$ facing an opponent of type $j$. Then
\[f_i(\bm{x})=\sum_{j=1}^n A_{i,j}x_j.
\] 
This  replicator equation, henceforth referred to as the {\it $\bm A$-replicator equation} may be written
\begin{equation}\label{originalrepeq}
 \dot x_i(t)=x_i(t)\big([{\bm A} {\bm x}(t)]_i-{\bm x}(t)^T{\bm A} {\bm x}(t)\big), \qquad i = 1,\cdots,k, \, t\geq0.
\end{equation}

If the system \eqref{originalrepeq} admits a fixed point in the simplex, i.e. $x_i(t) = x^*_i$, $i = 1, \cdots, k$, for some vector ${\bm x}^* = (x^*_1, \cdots, x^*_k)\in\mathcal{S}^k$, so that $\dd(\sum_{i=1}^k x_i(t))/\dd t=0$, then we see  that, necessarily,
 \begin{equation}
[\bm{A}\bm{x}^*]_i=(\bm{x}^*)^T\bm{A}\bm{x}^* \quad i =1, \cdots, k. 
\label{fixedpoint}
\end{equation}
In turn, this  implies that there is a constant, $c>0$, such that ${\bm A}{\bm x}^*=c\bm{1}$, so, ${\bm x}^*=c\bm{A}^{-1}\bm{1}$, where $\bm 1$ is the vector in $\mathbb{R}^k$ with unit entries. Since ${\bm x}^*\in\mathcal{S}^k$, it follows that $\bm{1}^T{\bm x}^*=1$, and hence   $c=(\bm{1}^TA^{-1}\bm{1})^{-1}$, thus \eqref{originalrepeq} is solved by 
\begin{equation}
{\bm x}^*=\dfrac{A^{-1}\bm{1}}{\bm{1}^TA^{-1}\bm{1}}.
\end{equation}

\medskip

If ${\bm x}^*$ satisfies the relation 
\[
({\bm x}^*)^T{\bm A}{\bm x} >{\bm x}^T{\bm A}{\bm x}, 
\]
for all $ {\bm x}\neq \bm x^*$ in a neighbourhood of $\bm x^*$,
then it is called an {\it evolutionary stable state} (ESS). Theorem 7.2.4 of \cite{HS} states that if ${\bm x}^*$ is a ESS, then 
\begin{equation}
\lim_{t\to\infty}{\bm x}(t) = {\bm x}^*.
\label{ESScgce}
\end{equation}

\medskip

The following results give us a remarkable connection between  the theory of replicator equations and  our coalescent model. To this end we define our $A$ matrix by
\begin{equation}
A_{i,j} = -\left(C_{j,i} \mathbf{1}_{j\neq i} + \frac{1}{2} C_{i,i}\mathbf{1}_{i=j}\right).
\label{A}
\end{equation}
For the remainder of the paper we will assume the rates $C$ are such that \eqref{ESScgce} holds. 

%
\medskip

Let us now state the connection between the notion of coming down from infinity for the replicator coalescent and the corresponding replicator equations.

\begin{theorem}\label{prereplicator} 
Suppose that $\bm A$ is such that \eqref{ESScgce} holds and that $(\bm \eta^N, N\geq 1)$ tends to $( \bm r_0, 0)$. Then for all $T>0$,
\[
\lim_{N\to\infty}\mathbb{E}_{{\bm \eta}^N}\left[\sup_{t\leq T}\norm{\bm R(t)- \bm x(t)}_1\right]=0,\qquad i =1,\cdots, k,
 \]
 where $\bm R(t) = \bm r(\tau(t))$, $t\geq0$, $\bm x(t) = (x_1(t), \cdots, x_k(t))$ solves the $\bm A$-replicator equation with initial condition $\bm x(0) = \bm r_0$ and 
 \[
 \tau(t)
= \inf\{s>0: \int_0^s \sigma(u)\dd u>t\}, \qquad t\geq 0.
 \]
 In particular, 
 \[
 \lim_{t\uparrow\infty}\lim_{N\to\infty}\mathbb{E}_{{\bm \eta}^N}[\norm{ \bm R(t)- \bm x^*}_1]=0,\qquad i =1,\cdots, k.
 \]
\end{theorem}
\begin{figure}[h!]
\vspace{-2cm}
  \begin{center}
\includegraphics[width=0.9\textwidth]{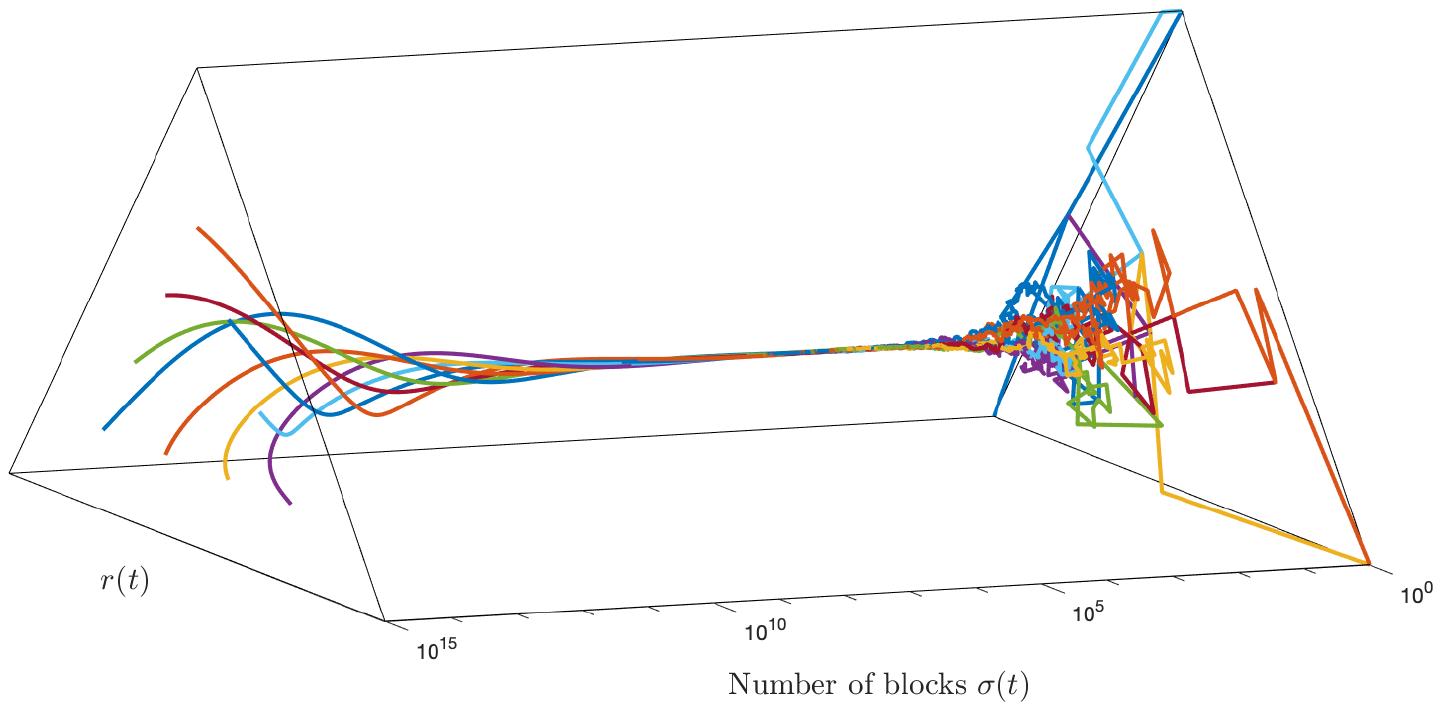}
  \end{center}
  \vspace{-1.5cm}
\caption{\it Simulations of a  replicator coalescent with $k = 3$ initiated from a variety of initial states with an initial number of blocks $\sigma (0)=10^{15}$. Each path represents a simulation from a different initial state, presented in barycentric coordinates in the 3-simplex and a logarithmic axis for the total number of blocks. The matrix $C$ has entries $C_{i,i}=C_{i,i+1}=1$ and other entries zero. The reader will note this case in particular demonstrates that we clearly do not need to enforce $C_{i,j}>0$ for all $i,j$.}
  \label{fig}
\end{figure}
In words, Theorem \ref{prereplicator} says that, by dilating time arbitrarily close to zero,  its process $\bm r$ in the simplex behaves deterministically like the solution to an $\bm A$-replicator equation. For a special choice of the matrix $C_{i,j}$, Figure \ref{fig} shows simulations of the process $(\bm R(t), t\geq0)$ which resonate with the statement of Theorem \ref{prereplicator}. As part of the proof of Lemma \ref{comingdownfrominfinity}, we will see that under $\mathbb{P}_{\bm \eta^N}$, the process $\sigma$ is comparable to a Kingman coalescent with some collision rate, say $c>0$. Noting that 
\[
\int_0^{\tau(t)}\sigma(u)\dd u = t, \text{ which implies  } \dot \tau(t) = 1/\sigma(\tau(t)),
\] 
   if, in heuristic terms, we treat $\sigma$ as a solution to \eqref{Kode} with $\sigma(0) = N$, so that $\sigma(t)^{-1} = N^{-1}+ ct/2$, we have that 
  $\tau(t)\approx 2(\exp\{ct/2\}-1)/cN\approx t/N$.

For a Kingman coalescent, say $(\nu(t), t\geq0)$ with merger rate $c$, classical reasoning tells us that $(\nu(t/N)/N, t\geq0)$  converges in an appropriate sense to the solution to   the ODE \eqref{Kode} as $\nu(0) = N\to\infty$. 
 In the spirit of these arguments, we would therefore expect that we can similarly control $(\sigma(t/N)/N, t\geq0)$  as well as $(\bm{n}(t/N)/N, t\geq0)$ as $N\to\infty$. In turn,  since we can write 
   \[
   \bm{R}(t)\approx \bm{r}(t/N) = \frac{\bm{n}(t/N)/N}{\sigma(t/N)/N},\qquad t\geq0,
   \] 
   we can therefore think of Theorem \ref{prereplicator} as a version of the aforementioned functional scaling result for Kingman's coalescent.

\medskip

The remainder of this paper is structured as follows. In the next section, we discuss how we can compare the process $(\sigma(t), t\geq0)$ with Kingman's coalescent on the same probability space, when it is issued from a finite number of blocks. This comparison  is used frequently in several of our proofs. 
In Section \ref{smgsect} we treat the Markov process $\bm n$ as a semimartingale and study its decomposition as the sum of a martingale and a bounded variation compensator. This provides the basis for the proofs of Theorem \ref{prereplicator}, which are given in Section \ref{prepostproofs}. Finally in Section \ref{remarks} we conclude with some technical remarks and two conjectures concerning further behaviour of the entrance law.

\section{Stochastic comparison with Kingman's coalescent}\label{proof1}
As alluded to above, there are various points in our reasoning where we will compare the number of blocks in a replicator coalescent with the number of blocks in an appropriately formulated Kingman coalescent on the same probability space. 
The first such result is reasoning  gives us the proof that the replicator coalescent comes down from infinity.

\begin{proof}[Proof of Lemma \ref{comingdownfrominfinity}]
The process $\sigma$ decreases by one with each block merger, analogously to the block counting process of a standard Kingman coalescent. It is therefore sufficient to prove that $\sigma$ decreases at least as slow, or at least as fast, as a Kingman block counting process with comparable rates.

\medskip

 Formally speaking, from finite starting states, we want to construct such a Kingman coalescent $\nu^+$ on the same space as $\sigma$ with that $\sigma\leq \nu^{+}$ by considering the minimal  rate of $\sigma$. 

\medskip

The total rate of mergers in state $\bm{n}$ is given by 
\begin{align}\label{rates}
\rho(\bm{n})=\sum_{i=1}^k\left(\sum_{j\neq i}C_{i,j}n_{j}n_i +C_{i,i} {n_i \choose 2}
\right)\,,
\end{align}
which depends not just on the total number of blocks, but also on the distribution of block types. However, since $C_{i,j}\geq0$  for all $i,j$, 
we can choose $C> \max_{i,j}C_{i,j}$ such that
\begin{equation}
\rho(\bm{n})
<
\sum_{i=1}^k\frac{1}{2}Cn_i\left(\sum_{j\neq i}n_{j} +  n_i-1\right) = 
\sum_{i=1}^k\frac{1}{2}Cn_i(\norm{\bm n}_1-1) = {C} {\norm{\bm n}_1 \choose 2}.
%
%
\label{rhoupper}
\end{equation}
Appealing to the skip free property, it follows that we can stochastically couple a Kingman coalescent death chain $(\nu^{-}(t), t\geq0)$ with collision rate $C$ and the process $({\bm n}(t), t\geq 0)$ on the same space such that, with  ${\bm \eta} = (\arg({\bm n}(0)), \sigma(0)^{-1})\in \mathcal{S}^k\times \mathbb{N}^{-1}$ and $\nu^-(0) = \sigma(0)$, $\sigma(t)\geq \nu^-(t)$, $t\geq0$.

\medskip


Conversely, when $C_{i,j}>0$ for all $i,j$, 
writing $\underline{C} = \min_{i,j}C_{i,j}$, we get 
\[
\rho(\bm{n})
\geq
\sum_{i=1}^k\frac{1}{2}\underline{C}n_i\left(\sum_{j\neq i}n_{j} +  n_i-1\right)
= \underline{C}{\norm{\bm n}_1\choose 2}.
\]
In the same spirit,  follows that we can stochastically embed another Kingman death chain $(\nu^+(t), t\geq0)$ and the process $({\bm n}(t), t\geq 0)$ on the same space such that, with  ${\bm \eta} = (\arg({\bm n}(0)), \sigma(0)^{-1})\in \mathcal{S}^k\times \mathbb{N}^{-1}$ and $\nu^-(0) = \sigma(0)$, $\sigma(t)\leq \nu^+(t)$, $t\geq0$.

In particular, as $\nu^+$ comes down from infinity, since $
\gamma_m \leq \beta^+_{m}:=\inf\{t>0:\nu^+(t)=m\}$, it follows that, for each $\varepsilon>0$,
\[
\lim_{m\to\infty}\lim_{N\to\infty}\mathbb{P}_{{\bm \eta}^N}[\gamma_m<\varepsilon ]\geq \lim_{m\to\infty}\lim_{N\to\infty}\mathbb{P}[\beta^+_m<\varepsilon| \nu^+(0) =N]=1, 
\]
thus completing the proof.
\end{proof}

\section{Semimartingale representation}\label{smgsect}
We would like to treat the replicator coalescent $(\bm n(t), t\geq0)$ as a semimartingale. It turns out to be more convenient to consider instead the vectorial process 
\[
\bm y(t)=\begin{bmatrix}\bm r(t\wedge \gamma_1)\\\\ {1}/{\sigma(t\wedge \gamma_1)}\end{bmatrix},\qquad t\geq 0,
\]
where 
\[
\gamma_1 = \inf\{t>0: \sigma(t) = 1\}.
\]
Naturally, by expressing the evolution of $(\bm y(t), t\geq0)$ as that of a semimartingale our interest is predominantly  in the process $(\bm r(t), t\geq0)$ so as  to make a link with the replicator equations in \eqref{originalrepeq}.

\medskip

\begin{lemma}\label{2} For each $\bm\eta\in\mathbb{N}^k_*$, 
the process $\bm y$ under $\mathbb{P}_{\bm\eta}$ has a semimartingale decomposition 
\begin{equation}
\bm y(t) = \bm y(0)+ \bm m(t) + \bm \alpha(t), \qquad t\geq0
\label{ysemimg}
\end{equation}
 where $(\bm m(t), t\geq0)$  is a martingale taking the form
\begin{equation}
\bm m(t) = \sum_{s\leq t\wedge \gamma_1}\Delta \bm y(s) - \bm\alpha(t), \qquad t\geq0,
\label{mgdef}
\end{equation} 
such that 
$
\Delta\bm y(t) =\bm y(t)- \bm y(t-) 
$
and $(\bm \alpha(t), t\geq0)$ is a compensator taking  the form 
\begin{equation}
\bm\alpha(t) =\int_{0}^{t\wedge \gamma_1 } \frac{\sigma(s)}{\sigma(s)-1}\sum_{i = 1}^k \begin{bmatrix} \sigma(s)({\bm r} (s)-\bm{e_i}) \\ \\ 1\end{bmatrix}r_i(s) [\sigma(s)^{-1} {\rm diag}(\bm A)\bm{1}-\bm A\bm r(s) ]_i \,\dd s.
\label{da}
\end{equation}
\end{lemma}
\begin{proof}
A standard computation using the compensation formula tells us that 
$\bm m$ is a martingale provided that $\sum_{s\leq t}\norm{\Delta \bm m(s)}_1=\sum_{s\leq t}\norm{\Delta \bm y(s)}_1$ has finite expectation for each $t\geq 0$, which is equivalent to the existence of the compensator $\bm \alpha(t)$. The latter is given by the rates that define the replicator coalescent. More precisely, recalling \eqref{A}, 
\begin{align}
\bm\alpha(t)
&=\int_{0}^{t\wedge \gamma_1}\sum_{i=1}^k \begin{bmatrix}\dfrac{\bm n(s)-\bm{e_i}}{\sigma(s)-1}-\dfrac{\bm n(s)}{\sigma(s)}\\ \\ \dfrac{1}{\sigma(s)-1}-\dfrac{1}{\sigma(s)}\end{bmatrix}\Bigg[ \sum_{\substack{j=1 \\ j\neq i}}^kn_j(s)n_i(s)C_{ji}+\frac{1}{2}n_i(s)(n_i(s)-1)C_{i,i}\Bigg] \dd s\notag\\
&=\int_{0}^{t\wedge \gamma_1}\sum_{i=1}^k\begin{bmatrix}\dfrac{\bm n(s)-\bm{e_i}\sigma(s)}{(\sigma(s)-1)\sigma(s)}\\\\ \dfrac{1}{(\sigma(s)-1)\sigma(s)}\end{bmatrix}\Bigg[ n_i(s)A_{i,i}-\sum_{j=1 }^kn_j(s)n_i(s)A_{i,j} \Bigg] \dd s\notag\\
&=\int_{0}^{t\wedge \gamma_1} \dfrac{\sigma(s)}{\sigma(s)-1}\sum_{i = 1}^k \begin{bmatrix} \sigma(s)(\bm r (s)-\bm{e_i}) \\ \\1\end{bmatrix}r_i(s) [\sigma(s)^{-1} {\rm diag}(\bm A)\bm{1}-\bm A\bm r(s) ]_i \,\dd s,
\label{deriveda}
\end{align}
as required. Note, if we identify the compensator via the density $(\bm\lambda(t), t\geq0)$, where  
\begin{equation}
\bm\alpha(t) = :\int_0^{t\wedge \gamma_1} \bm\lambda(s)\dd s
\label{lambda}
\end{equation}
 then, in the  above representation of $\bm\alpha$, the largest term is of order $\sigma(s)$, from which, because the process $\bm n$ is non-increasing, we can  easily conclude that,  for all $\bm\eta\in\mathbb{N}^k_*$ and any  time $t\geq 0$, 
\begin{equation}
\mathbb{E}_{\bm \eta^N}\left[ \sum_{s\leq t \wedge \gamma_1}\norm{\Delta\bm y(s)}_1\right]\leq 
 \mathbb{E}_{\bm \eta}\left[\int_0^{t\wedge \gamma_1} \norm{\bm\lambda (s)}_1 \dd s\right]\leq C\norm{\bm\eta}_1\mathbb{E}_{\bm \eta^N}[t\wedge \gamma_1]\leq C\norm{\bm\eta}_2t,
\label{mgcontrol}
\end{equation}
for an unimportant constant $C>0$. This
 ensures that  $\bm m$ is a martingale and that $\bm\alpha$ is well defined. 
\end{proof}

For the proof of Theorem \ref{prereplicator}, we are interested in the  behaviour of the process $\bm r$  under $\mathbb{P}_{\bm\eta}$ for any $\bm\eta$ such that  $\norm{\bm\eta}_1\to\infty$ and $\arg(\bm\eta)\to\bm r$ for some $\bm r\in\mathcal{S}^k_+$. Heuristically speaking, the term $\sigma(s)({\bm r} (s)-\bm{e_i})$ in the expression for $\bm\alpha$ suggests that  $\alpha(t)$ explodes  as $t\to 0$. The undesirable factor $\sigma(s)$ can be removed however by an appropriate time change and in doing so, we begin to see where the relationship with the replicator equations emerges.

\begin{lemma}
\label{3} Suppose we define the sequence of stopping times  $(\tau(t), 
t\geq 0)$, which are defined by the right inverse, 
\begin{equation}
\tau(t)
= \inf\{u>0: \int_0^u \sigma(s)\dd s>t\}, \qquad t\geq 0.
\label{taudef}
\end{equation}
Then 
$\bm y^\tau: = \bm y\circ\tau$ has semimartingale  decomposition 
$\bm y^\tau = \bm  m^\tau + \bm\alpha^\tau $, where $ \bm m^\tau:= \bm m \circ \tau$ is a martingale and, for 
$t\geq0$,
\begin{align}
\bm\alpha^\tau(t)  
&=\int_{0}^{t \wedge \tau^{-1}(\gamma_1)} \dfrac{\sigma(\tau(s))}{\sigma(\tau(s))-1}\notag\\
&\hspace{2cm}\sum_{i = 1}^k \begin{bmatrix} \bm r (\tau(s))-\bm{e_i} \\ \\\dfrac{1}{\sigma(\tau(s))}\end{bmatrix} r_i(\tau(s)) [ \sigma(\tau(s))^{-1} {\rm diag}(\bm A)\bm{1}- \bm A\bm r(\tau(s)) ]_i \,\dd s
\label{alphatau}
\end{align}
\end{lemma}
\begin{proof}
We use basic Stieltjes calculus  to tell us that 
\begin{equation*}
\dd \bm\alpha^\tau(t) = \dd \bm\alpha(s)|_{s = \tau(t)}\dd \tau(t).
\end{equation*}
Moreover, 
\begin{equation}
\int_0^{\tau(t)} \sigma(s)\dd s = t \,\,\,\text{ and hence } \,\,\, \sigma(\tau(t))\dd \tau(t) = \dd t 
\label{dt}.
\end{equation}
Combining these observations with the conclusion of  Lemma \ref{2} the result follows. 
In particular, from \eqref{lambda},
\begin{equation}
\int_0^{\tau(t)\wedge \gamma_1} \bm\lambda(s)\dd s = \int_0^{t \wedge\tau^{-1}(\gamma_1)} \frac{\bm\lambda(\tau (u))}{\sigma(\tau(u))}\dd u.
\label{dt2}
\end{equation}
Technically, we need to verify that $\bm m^\tau$ is a martingale rather than a local martingale, however a computation similar to \eqref{mgcontrol} taking advantage of \eqref{alphatau} is easily executed affirming the required martingale status. 
\end{proof}

Next we look at how to control the second moment of the martingale $\bm m^\tau$ in the semimartingale decomposition of $\bm y^\tau$.

\begin{lemma}\label{2upper}
For each $\bm\eta \in\mathbb{N}^k_*$, the martingale $\bm m^\tau$ under $\mathbb{P}_{\bm \eta}$ satisfies
\[
\mathbb{E}_{\bm \eta^N}\left[\norm{\bm m^\tau(t)}^2_2\right] \leq C\mathbb{E}_{\bm\eta}\left[ \int_0^{\tau(t  ) \wedge\gamma_1}\frac{\sigma(s)^2}{(\sigma(s)-1)^2} \dd s\right] , \qquad t\geq0.
\]
\end{lemma}

\begin{proof}
Steiltjes calculus, or equivalently general semi-martingale calculus (see for example Theorem II.33 of \cite{Protter}), tells us that, since $\bm m$ has bounded variation,
\begin{align}
\norm{\bm m (t)}^2_2 &= 2\int_0^{t\wedge \gamma_1} \bm m(s-)\cdot\dd\bm m(s) +  \sum_{0<s\leq t \wedge \gamma_1}\left\{\Delta  
\norm{\bm m (s)}_2^2 - 2\bm m(s-)\cdot\Delta \bm m(s)  \right\}\notag\\
&=2\int_0^{t\wedge \gamma_1} \bm m(s-)\cdot\dd\bm m(s) +  \sum_{0<s\leq t \wedge \gamma_1}(\Delta \bm m(s))^2.
\label{prot}
\end{align}
As all vector entries are bounded, it is easy to show that 
$\textstyle\int_0^{t\wedge \gamma_1} \bm m(s-)\cdot\dd\bm m(s)$, $t\geq0$, is a martingale. 

Next, we identify the adapted increasing bounded variation process, say $\bm\beta$, which is the compensator of $ \sum_{0<s\leq t \wedge \gamma_1}(\Delta \bm m(s))^2$, $t\geq0$ so that 
\[
\norm{\bm m (t)}_2^2 -\bm\beta(t), t\geq 0
\] 
is a martingale with mean  0.

To this end, 
note that $\Delta \bm m(t) = \Delta \bm y(t)$. 
Hence, we have, on the event that $t$ is a time at which the number of blocks of type $i$ decreases,  $(\Delta \bm m(t))^2$ is given by 
\begin{align*}
\chi_i(t)
&:= 
\begin{bmatrix}\dfrac{\bm n(t)-\sigma(t)\bm{e_i}}{(\sigma(t)-1)\sigma(t)}  \\ \\ \dfrac{1}{(\sigma(t)-1)\sigma(t)}\end{bmatrix}
\cdot
\begin{bmatrix}\dfrac{\bm n(t)-\sigma(t)\bm{e_i}}{(\sigma(t)-1)\sigma(t)}  \\ \\ \dfrac{1}{(\sigma(t)-1)\sigma(t)}\end{bmatrix}
\notag\\
&=\frac{\bm n(t)\cdot\bm n(t) -2\sigma(t)  n_i(t) + \sigma(t)^2+1}{\sigma(t)^2(\sigma(t)-1)^2}.
\end{align*}

It is now straightforward to see that, there exists a $C>0$ such that 
\begin{align*}
\bm \beta(t) &=\int_0^{t \wedge \gamma_1}\sum_{i=1}^k
 \chi_i(s)\Bigg[ \sum_{\substack{j=1 \\ j\neq i}}^kn_j(s)n_i(s)C_{ji}+\frac{1}{2}n_i(s)(n_i(s)-1)C_{i,i}\Bigg]\dd s\\
&\leq  C \int_0^{t\wedge \gamma_1}\frac{\sigma(s)^2}{(\sigma(s)-1)^2} \dd s,
\end{align*}
where we have used that \[
\sigma(t)^2 = (n_1(t) +\cdots+n_k(t))^2\geq  \bm{n}(t)\cdot \bm{n}(t) .
\]
Replacing $t$ by $\tau(t)$ and taking expectation gives the desired inequality.
\end{proof}

As the reader may now expect, our ultimate objective is to show that for any sequence of starting initial configurations $\bm \eta^N\in \mathbb{N}^k_*$ such that $\norm{\bm \eta^N}_1\to\infty$ as $N\to\infty$ and $\arg(\bm \eta^N)\to \bm r\in \mathcal{S}^k_+$, the   martingale component $\bm m^\tau$ disappears. This tells us that the behaviour of $(\bm r(t), t\geq0)$ behaves increasingly like the compensator term, which is further key to controlling its behaviour.  To this end, we conclude this section with two more results that provide us with the desired control of the aforesaid martingale component.

\begin{lemma}\label{cor1} Fix $t>0$. Suppose $(\bm \eta^N, N\geq 1)$ tends to $( \bm r_0, 0)$,  
\[
\tau^{-1}(t) =   \int^{t}_0 \sigma(s)\dd s \to \infty, \, \,  \tau(t)\to 0,\, \, \text{ and }\,\, |\tau(t)\wedge \gamma_1-\tau(t)|\to0,
\]
weakly as $N\to\infty$.
\end{lemma}

\begin{proof}
Recall from the proof of Lemma \ref{comingdownfrominfinity}, and specifically \eqref{rhoupper},  there is a death chain $(\nu(t), t\geq0)$, representing the number of blocks in a Kingman coalescent with merger rate $C$  such that, for any $\bm \eta\in\mathbb{N}^k_*$, on the same probability space, we can stochastically bound $\sigma(t)\geq \nu(t)$, $t\geq0$.

\medskip

We now note that, for any large $M>0$ there exists a constant $C>0$ (not necessarily the same as before) such that, for any $m$ sufficiently large,
\begin{align}
\lim_{N\to\infty}\mathbb{P}_{{\bm \eta}^N}\left(\int_0^{\gamma_m} \sigma(s)\dd s> M\right) 
&\geq \lim_{N\to\infty}\mathbb{P}_{{\bm \eta}^N}\left(\int_0^{\beta_m} \nu(s)\dd s> M\right) \notag\\
&= \Pr\left(\sum_{n = m+1}^\infty \frac{n}{C{n\choose 2}}\mathbf{e}^{(n)}_{1} >M\right)\notag\\
&=\Pr\left(\sum_{n = m+1}^\infty \frac{1}{n-1}\mathbf{e}^{(n)}_{1} >\frac{CM}{2}\right),
\label{eq1}
\end{align}
where $\beta_m = \inf\{t>0: \nu(t) = m\}$ and $({\bf e}^{n}_1, n\geq 1)$, is a sequence of iid unit-mean exponentially distributed random variables.
If we write $(N(t), t\geq 0)$ for a Poisson process with unit rate, then almost surely we have
\[
\sum_{n = m+1}^\infty \frac{1}{n-1}\mathbf{e}^{(n)}_{1} = \int_0^\infty \frac{1}{N(s)+m}\dd s 
=\int_0^\infty \frac{s}{N(s)+m}\frac{\dd s}{s} =\infty
\]
where  the final equality follows by the strong law of large numbers for Poisson processes. As such the right-hand side of \eqref{eq1} is equal to 1.

\medskip

Since $M$ and $m$ can be arbitrarily large, this shows the first claim as soon as we note that $\tau^{-1}(t) = \int_0^t \sigma(u)\dd u$, which  is an easy consequence of  \eqref{dt}. On the other hand, note that since 
$
\int_0^{\tau(t)} \sigma(s)\dd s = t,
$
when $\sigma(0)= {\bm \eta}^N \rightarrow\infty$, the above comparison with Kingman's coalescent shows that  almost surely that, since $\int_0^u\sigma(s)\dd s $ converges weakly to infinity for all $u>0$, then $\tau(t)$ converges weakly to 0. Indeed, if with positive probability $\tau(t)>\varepsilon$ in the limit as $N\to\infty$, then, on that event, 
$\int_0^{\tau(t)} \sigma(s)\dd s\geq \int_0^{\varepsilon} \sigma(s)\dd s$, which explodes in distribution. This in turn contradicts the definition of $\tau(t)$. This proves the second and third statements of the lemma.
\end{proof}

%
%


Since the second moment of the martingale $\bm m^\tau$ can be controlled by its associated time change, we also get a helpful $L^2$ corollary from Lemma \ref{cor1}.
\begin{cor}\label{cor2}
From Lemma \ref{2upper} and Lemma \ref{cor1}, we deduce that, if $(\bm \eta^N, N\geq 1)$ tends to $( \bm r_0, 0)$, then, for each $t>0$
\[
\lim_{N\to\infty}\mathbb{E}_{\bm \eta^N}\left[\sup_{s\leq t}\norm{\bm m^\tau(s) }^2_2\right] =0
\]
\end{cor}
\begin{proof}
We have from Lemma \ref{2upper} and a change of variable similar to \eqref{dt2}
that
\begin{align*}
\mathbb{E}_{\bm \eta^N}\left[\norm{\bm m^\tau(t) }^2_2\right] &\leq C\mathbb{E}_{\bm\eta^N}\left[ \int_0^{\tau(t  )\wedge\gamma_1}\frac{\sigma(s)^2}{(\sigma(s)-1)^2} \dd s\right]\notag\\
& = C\mathbb{E}_{\bm\eta^N}\left[ \int_0^{t\wedge\tau^{-1}(\gamma_1)}\frac{\sigma(\tau(u))^2}{(\sigma(\tau(u))-1)^2}\frac{1}{\sigma(\tau(u))} \dd u\right]\notag\\
& \leq C t\mathbb{E}_{\bm \eta^N}\left[\frac{1}{\sigma(\tau(t)) }\right].
\end{align*}
We can choose $N$ sufficiently large such that $\tau(t)$ is less than $\delta$ with probability at least $1-\epsilon$ so that, 
\begin{align*}
\mathbb{E}_{\bm \eta^N}\left[\frac{1}{\sigma(\tau(t))} \right]&\leq \mathbb{E}_{\bm \eta^N}\left[\frac{1}{\sigma(\delta)}  ; \tau(t)< \delta\right] +  \mathbb{P}_{\bm \eta^N}(\tau(t)\geq \delta)\notag \\
&\leq\delta \mathbb{E}_{\bm \eta^N}\left[\frac{1}{\delta\nu(\delta)}\right] + \epsilon,
\end{align*}
where we have again compared with a lower bounding Kingman coalescent $(\nu(t), t\geq0)$ on the same space, as in Lemma \ref{cor1}. Recall the classical result for Kingman's coalescent coming down from infinity that, when the collision rate is $C>0$,
\[
\delta\nu(\delta)\to 2/C
\]
 almost surely as $\delta\to 0$; cf \cite{Brazil}.  We can now easily conclude with the help of dominated convergence that 
\[
\lim_{N\to\infty}\mathbb{E}_{\bm \eta^N}\left[\norm{\bm m^\tau(t)}^2_2\right]  = 0,
\]
and this concludes the proof once we invoke Doob's submartingale inequality. 
\end{proof}

\section{Proof of Theorem \ref{prereplicator}}\label{prepostproofs}
Recall that we have required from the $\bm A$-replicator equation  that  $\bm x(t)\to \bm x^*$ holds. Reinterpreting \eqref{originalrepeq} in its integral form, this tells us that 
\begin{equation}\label{mild}
x_i(t) = x_i(0) +\int_0^t x_i(s)([{\bm A} {\bm x}(s)]_i-{\bm x}(s)^T{\bm A} {\bm x}(s))\dd s, \qquad t\geq0.
\end{equation}
This representation makes it easier to give the heuristic basis of the proof of Theorem \ref{prereplicator}.

\medskip

Following our earlier heuristic reasoning, we can now see that, under $\mathbb{P}_{\bm\eta}$ as $\norm{\bm\eta}_1\to\infty$ the integrand in the expression for $\bm\alpha^\tau$ appears to have a similar structure to the replicator equations \eqref{originalrepeq}. That is, under $\mathbb{P}_{\bm\eta}$ as $\norm{\bm\eta}_1\to\infty$ and as $t\to 0$,
\[
\frac{\dd \bm\alpha^\tau(t)}{\dd t}\approx 
\begin{bmatrix}\bm\theta(t)\\ \\ 0 \end{bmatrix}
\] 
where
\[
\bm\theta_i(t) = { r}_i(\tau(t) )\Big([\bm A\bm r(\tau(t))]_i  -\bm r(\tau(t))^T\bm A\bm r(\tau(t)))       \Big).
\]
On the other hand, if we can show that $\bm m^\tau\to0$ as $\norm{\bm\eta}_1\to\infty$, then, since $\bm y^\tau = \bm  m^\tau + \bm\alpha^\tau $, 
reading off the first component of $\bm y^\tau$, i.e. $\bm R(t): = \bm r(\tau(t))$, roughly speaking we see that 
\[
R_i(t) \approx  r_i(0)+  \int_0^t  R_i(s )\Big([\bm A\bm R(s)]_i  -\bm R(s)^T\bm A\bm R(s))       \Big)\dd s,
\]
as $\norm{\bm\eta}_1\to\infty$, given that Corollary \ref{cor2} shows the martingale component is negligible. In other words, the process $(\bm R (t), t\geq0)$, begins to resemble the replicator equation in its integral form \eqref{mild}. It now looks like a reasonable conjecture that $\bm R(t)\to \bm x^*$, just as the solution to the replicator equation does. 

\medskip

Let us thus move to the proof of Theorem \ref{prereplicator}, which, as alluded to earlier, boils down to the control we have on  the martingale $\bm m^\tau$ under $\mathbb{P}_{\bm\eta}$ as $\norm{\bm \eta}_1\to\infty$, thanks to Corollary \ref{cor2}.

\begin{proof}[Proof of Theorem \ref{prereplicator}]
Write $\bm R_i(t) = \bm r(\tau(t))$ on the event  
\[
A_t := \{t<\tau^{-1}(\gamma_1) \},\qquad  t\geq0,
\]
 and note from Lemma \ref{cor1}  and the proof of Lemma \ref{comingdownfrominfinity} that $\lim_{N\to\infty}\mathbb{P}_{\bm \eta^N}(A_t)=1$. We have, for each $T>0$,
\begin{align}
&\mathbb{E}_{\bm \eta^N}\Bigg[\sup_{t\leq T}\Bigg\|
{\bm R}(t) - {\bm r}(0)- \int_0^t \Big(\sum_{i=1}^k {\bm e}_iR_i(s)[\bm A\bm R(s)]_i  - (\bm R(s)^T\bm A\bm R(s) ) \bm R(s)      \Big)\dd s  \Bigg\|_1 \mathbf{1}_{A_t}\Bigg]
\notag\\
&\leq \mathbb{E}_{\bm \eta^N}\Bigg[\sup_{t\leq T} \Bigg\|{\bm R}(t) - {\bm r}(0)- \int_0^t \frac{\sigma(\tau(s))}{\sigma(\tau(s))-1}
\notag\\
&\hspace{3.5cm}\sum_{i = 1}^k 
(\bm R (s)-\bm{e_i} )   R_i(s) [ \sigma(\tau(s))^{-1} {\rm diag}(\bm A)\bm{1}- \bm A\bm R(s) ]_i \,\dd s
\Bigg\|_1\mathbf{1}_{A_t}\Bigg] \notag\\
&+ \mathbb{E}_{\bm \eta^N}\Bigg[\sup_{t\leq T} \Bigg\|  \int_0^t \frac{1}{\sigma(\tau(s))-1}\Big(\sum_{i=1}^k {\bm e}_iR_i(s)[\bm A\bm R(s)]_i  - (\bm R(s)^T\bm A\bm R(s) ) \bm R(s)      \Big)\,\dd s
\Bigg\|_1\mathbf{1}_{A_t}\Bigg]\notag\\
&+ \mathbb{E}_{\bm \eta^N}\Bigg[\sup_{t\leq T} \Bigg\|\int_0^t \frac{1}{\sigma(\tau(s))-1}
\Big(
\bm ({\bm R}(s)^T{\rm diag}(\bm A) \bm 1)\bm R(s)
- \sum_{i=1}^k {\bm e}_iR_i(s)[{\rm diag}(\bm A) \bm 1]_i
\Big)
\Bigg\|_1\mathbf{1}_{A_t}\Bigg].
\label{whopper}
\end{align}
From Corollary \ref{cor2}, and the fact that $\norm{\bm y}_1\leq \sqrt{k}\norm{\bm y}_2$, for $\bm y\in \mathbb{R}^k$, the first term after the inequality tends to zero as $N\to\infty$. Up to a multiplicative constant, the second and third terms after the inequality can be bounded by 
\[
T\mathbb{E}_{\bm \eta^N}\Bigg[\frac{1}{(\sigma(\tau(T)) -1)}\wedge1\Bigg],
\]
where we have used the monotonicity of $\tau(\cdot)$ and $\sigma(\cdot)$.
As noted in the proof of Corollary \ref{cor2}, the latter tends to zero as $N\to\infty$.

It now follows that 
\begin{equation}
\lim_{N\to\infty}\mathbb{E}_{\bm \eta^N}\Bigg[\sup_{t\leq T}\Bigg\|
{\bm R}(t) - {\bm r}(0)- \int_0^t \sum_{i=1}^k {\bm e}_iR_i(s) [\bm A\bm R(s)]_i  - (\bm R(s)^T\bm A\bm R(s)  )\bm R(s)  \dd s  \Bigg\|_1\mathbf{1}_{A_t}\Bigg]=0.
\label{newnormcgce}
\end{equation}

As all of the vectorial and matrix terms in \eqref{newnormcgce} are bounded, it is also easy to see with the help of Lemma \ref{cor1} that 
\begin{align*}
&\lim_{N\to\infty}\mathbb{E}_{\bm \eta^N}\Bigg[\sup_{t\leq T}\Bigg\|
{\bm R}(t) - {\bm r}(0)- \int_0^t \sum_{i=1}^k {\bm e}_iR_i(s) [\bm A\bm R(s)]_i  - (\bm R(s)^T\bm A\bm R(s)  )\bm R(s)  \dd s  \Bigg\|_1\mathbf{1}_{A_t^c}\Bigg]\notag\\
&\leq \lim_{N\to\infty}C T \mathbb{P}_{\bm \eta^N}(A_T^c)=0,
\end{align*}
for some constant $C>0$, which gives us 
\begin{equation}
\lim_{N\to\infty}\mathbb{E}_{\bm \eta^N}\Bigg[\sup_{t\leq T}\Bigg\|
{\bm R}(t) - {\bm r}(0)- \int_0^t \sum_{i=1}^k {\bm e}_iR_i(s) [\bm A\bm R(s)]_i  - (\bm R(s)^T\bm A\bm R(s)  )\bm R(s)  \dd s  \Bigg\|_1\Bigg]=0.
\label{newnormcgce2}
\end{equation}

Consider the replicator equation initiated from any $\bm r(0)\in \mathcal{S}^k_+$. Similarly to \eqref{mild}, albeit in vectorial form, we can write the solution to \eqref{originalrepeq}, when issued from $\bm r(0)$ as
\begin{equation}
\bm x(t) - \bm r(0) - \int_0^t\sum_{i=1}^k {\bm e}_i x_i(s)[{\bm A} {\bm x}(s)]_i-({\bm x}(s)^T{\bm A} {\bm x}(s))\bm x(s)\dd s=0.
\label{mildform}
\end{equation}
Subtracting \eqref{mildform} off  in  \eqref{newnormcgce2}, we see that, for each $T>0$,
\begin{align}
& \mathbb{E}_{\bm \eta^N}\left[  \sup_{t\leq T}\norm{\bm R(t)- \bm x(t)}_1\right]\notag\\
&\leq  \mathbb{E}_{\bm \eta^N}\Bigg[\sup_{t\leq T}\Bigg\|
{\bm R}(t) - { \bm r}(0)- \int_0^t \sum_{i =1}^k \bm  e_iR_i(s)\Big([\bm A\bm R(s)]_i  - \bm R(s)^T\bm A\bm R(s)    \Big)\dd s  \Bigg\|_1 \Bigg]\notag\\
&+\int_0^T \sum_{i =1}^k 
 \mathbb{E}_{\bm \eta^N}\left[ 
|R_i(s)- x_i(s)| \Big|[\bm A\bm R(s)]_i  - \bm R(s)^T\bm A\bm R(s)    \Big| 
  \right]\dd s\notag\\
&+\int_0^T \sum_{i =1}^k 
x_i(s)  \mathbb{E}_{\bm \eta^N}\left[ 
\Big|[\bm A(\bm R(s)-\bm x(s))]_i  - \bm R(s)^T\bm A(\bm R(s)-\bm x(s))    \Big| 
\right] \dd s\notag\\
&+\int_0^T \sum_{i =1}^k 
x_i(s)
 \mathbb{E}_{\bm \eta^N}\left[ 
 \Big| (\bm R(s)-\bm x(s))^T\bm A\bm x(s)    \Big| 
\right]\dd s .
\label{alongthelinesof}
\end{align}
Noting that $0\leq R_i(s), x_i(s) \leq 1$ for all $i =1\cdots, k$, $s\geq 0$, we have, 
\begin{align}
& \mathbb{E}_{\bm \eta^N}\left[ \sup_{t\leq T} \norm{\bm R(t)- \bm x(t)}_1 \right]\notag\\
&\leq   \mathbb{E}_{\bm \eta^N}\Bigg[\sup_{t\leq T}\Bigg\|
{\bm R}(t) - { \bm r}(0)- \int_0^t \sum_{i =1}^k \bm  e_iR_i(s)\Big([\bm A\bm R(s)]_i  - \bm R(s)^T\bm A\bm R(s)    \Big)\dd s \Bigg\|_1 \Bigg]\notag\\
&\hspace{2cm}+C\int_0^T \mathbb{E}_{\bm \eta^N}\left[ \sup_{u\leq s}\norm{\bm R(u)- \bm x(u)}_1 \right] \dd s,
\label{pregron}
\end{align}
where $C>0$ is an unimportant constant.
Hence, using \eqref{newnormcgce2}, the monotonicity of norms and dominated convergence,
\begin{align*}
u(T)&:=\lim_{N\to\infty}\sup_{N'\geq N} \mathbb{E}_{\bm \eta^N}\left[  \sup_{t\leq T}\norm{\bm R(t)- \bm x(t)}_1 \right]\notag\\
&\leq C\int_0^T\lim_{N\to\infty}\sup_{N'\geq N} 
\mathbb{E}_{\bm \eta^N}\left[  \sup_{u\leq s}\norm{\bm R(u)- \bm x(u)}_1 \right] \dd s\notag\\
& = C\int_0^T u(s)\dd s,
\end{align*}
where $C>0$ is an unimportant constant.
Gr\"onwall's Lemma now tells us that 
\begin{equation}
\lim_{N\to\infty}\mathbb{E}_{\bm \eta^N}\left[ \sup_{t\leq T} \norm{\bm R(t)- \bm x(t)}_1 \right] = 0.
\label{onAt}
\end{equation}

By taking $t\to\infty$, we easily deduce that 
\[
\lim_{t\to\infty}\lim_{N\to\infty}
\mathbb{E}_{\bm \eta^N}\left[  \norm{\bm R(t)- \bm x^* }_1 \right]
\leq \lim_{t\to\infty}\lim_{N\to\infty}\mathbb{E}_{\bm \eta^N}\left[  \norm{\bm R(t)- \bm x(t)}_1 \right]
+\lim_{t\to\infty}\norm{\bm x(t)-\bm x^*}_1 =0.
\]
This completes the proof of the theorem.
\end{proof}

\section{Concluding remarks}\label{remarks}

For the purpose of the following discussion, we can assume that $C_{i,j}>0$ for all $i,j\in\{1,\cdots,k\}$.  
In further work, one may pursue the issue of Skorokhod continuity with respect to any  `entrance laws at infinity' that the process can come down from.

\medskip
Suppose  $\mathbb{D}$ is the space of   c\`adl\`ag paths from $[0,\infty)$ to $\mathbb{N}^k_*$, with  $\mathcal{D}$ as the Borel sigma algebra on $\mathbb{D}$ generated from the usual Skorokhod metric. 
A Markovian definition of coming down from infinity would require the existence of a
law  $\mathbb{P}^\infty$ on $(\mathbb{D}, \mathcal{D})$
which is consistent with $\mathbb{P}$ in the sense that, 
\[
\mathbb{P}^\infty(\bm n(t+s) = \bm n)
= \sum_{\substack{\bm n'\in \mathbb{N}^k_*\\ \norm{\bm n'}_1\geq \norm{\bm n}_1}} \mathbb{P}^\infty(\bm n(t) = \bm n') \mathbb{P}_{\bm n'}(\bm n(s) = \bm n)
, \qquad s, t>0, \bm n\in \mathbb{N}^k_*,
\]
 $\mathbb{P}^\infty(\sigma(t)<\infty)=1$, for all $t>0$, and $\mathbb{P}^\infty(\lim_{t\downarrow0}\sigma(t) = \infty)=1$.
 
 \medskip
 
As there is no single point on the boundary of $\mathbb{N}^k_*$ that represents an appropriate `infinity' to come down from,
one of the  associated issues is whether a unique entrance law exists or whether e.g. there is an entrance law of $( \bm r, \sigma^{-1})$ for each `infinity'  of  the form $(\bm r_0, 0)$, where $\bm r_0\in\mathcal{S}^k_+$. We believe the latter to hold.
\begin{conjecture}\label{conj0} An entrance law, say  $\mathbb{P}^{(\bm r_0, 0)}$, exists for each  $\bm{r}_0\in\mathcal{S}^k_+$. 
\end{conjecture}

 There is also the question of how we can see these different entrance laws in terms of the behaviour of the process at arbitrarily small times.
The following conjecture suggests that looking backwards in time, it will be difficult to differentiate between the different entrance laws proposed in Conjecture \ref{conj0}.

\medskip

\begin{conjecture}\label{conj1}Suppose $(\bm \eta^N, N\geq 1)$ tends to $( \bm r_0, 0)$ for some $\bm r_0\in\mathcal{S}^k_+$. We have
\[
\lim_{m\to\infty} \lim_{N\to\infty}\mathbb{E}_{{\bm \eta}^N}\left[\norm{ \bm r(\gamma_m)-\bm x^*}_1\right]=0, 
 \]
where we recall
 $
 \gamma_m=\inf\{t>0: \sigma(t) \leq m\},
 $ for $m\geq1$.
\end{conjecture}

In contrast to Theorem \ref{prereplicator}, Conjecture \ref{conj1} claims that, moving backwards through time  towards the instantaneous event at which the replicator coalescent comes down from infinity at the origin of time, the process $\bm r$ necessarily approaches  $\bm x^*$. As such, working backwards in time, the replicator coalescent never gets to see the `initial state' $(\bm r_0,0)$, from which its entrance law is constructed.

\medskip

Put together Theorems \ref{prereplicator} and Conjecture \ref{conj1} are really claiming  that $\bm x^*$ is a `bottleneck' for the replicator coalescent. Figure 1 simulates an example where $k =3$, in which the bottleneck phenomenon can clearly be seen.

\medskip
Theorem \ref{prereplicator} also shows that  under $\mathbb{P}_{\bm\eta^N}$ as $\bm \eta^N\to(\bm r_0,0)$, in an arbitrarily small amount of time (on the  natural time scale of the original Markov process), the process $\bm r$ will effectively jump from $\bm r_0$ to $\bm x^*$. Taking Conjecture \ref{conj1} into account, it is for this reason we pose our final conjecture, below.

\begin{conjecture}\label{conj2} Suppose $(\bm \eta^N, N\geq 1)$ tends to $( \bm r_0, 0)$, for some $\bm r_0\in\mathcal{S}^k_+$. Then, $\lim_{N\to\infty}\mathbb{P}_{\bm\eta^N} \to \mathbb{P}^{(\bm r_0, 0)}$ continuously on $(\mathbb{D}, \mathcal{D})$ if and only if  $\bm r_0 = \bm x^*$. 
\end{conjecture}

\section*{Acknowledgements}
LP was visiting TR and AEK as part of a doctoral visiting programme sponsored by the Internationalisation Research Office at the University of Bath. She would like to thank the university for their support. TR would like to thank Chris Guiver for useful discussions. AEK and TR acknowledge EPSRC grant support from  EP/S036202/1. Referees provided some extremely helpful feedback leading to the improvement of an earlier version of this document. AEK would also like to thank Jon Warren and Jason Schweinsberg for insightful discussion.

\bibliography{references}{}
\bibliographystyle{plain}




\end{document}